\documentclass[12pt]{article}
\usepackage{fullpage}
\usepackage{amssymb, amsmath, amsthm, enumerate, math tools}

\theoremstyle{plain}
\newtheorem{thm}{Theorem}[section]
\theoremstyle{definition}

\newtheorem{cor}[thm]{Corollary}
\newtheorem{lemma}[thm]{Lemma}

\newtheorem{innercustomthm}{Theorem}
\newenvironment{customthm}[1]
  {\renewcommand\theinnercustomthm{#1}\innercustomthm}
  {\endinnercustomthm}
\begin{document}
\title{On factorizations into coprime parts}
\author{Matthew Just and Noah Lebowitz-Lockard}
\date{}
\maketitle

\begin{abstract} Let $f(n)$ and $g(n)$ be the number of unordered and ordered factorizations of $n$ into integers larger than one. Let $F(n)$ and $G(n)$ have the additional restriction that the factors are coprime. We establish the asymptotic bounds for the sums of $F(n)^\beta$ and $G(n)^\beta$ up to $x$ for all real $\beta$ and the asymptotic bounds for $f(n)^\beta$ and $g(n)^\beta$ for all negative $\beta$.
\end{abstract}

\section{Introduction}

Let $f(n)$ and $g(n)$ be the number of unordered and ordered factorizations of $n$ into integers larger than $1$. In 1927, Oppenheim \cite{O} found an asymptotic formula for the sum of $f(n)$, namely
\[\sum_{n \leq x} f(n) \sim \frac{1}{2\sqrt{\pi}} \frac{x\exp(2\sqrt{\log x})}{(\log x)^{3/4}},\]
which Szekeres and Tur{\' a}n \cite{SzTu} rediscovered a few years later. Soon afterwards, Kalm{\' a}r \cite{K} proved that
\[\sum_{n \leq x} g(n) \sim -\frac{1}{\rho\zeta'(\rho)} x^\rho,\]
where $\zeta$ is the Riemann zeta function and $\rho \approx 1.73$ is the unique solution in $(1, \infty)$ to $\zeta(\rho) = 2$. More precisely, Hwang \cite{Hw} proved that
\[\sum_{n \leq x} g(n) = -\frac{1}{\rho \zeta'(\rho)} x^\rho + O(x^\rho \exp(-c (\log_2 x)^{(3/2) - \epsilon}))\]
for a positive constant $c$ and all positive $\epsilon$. (We use $\log_k x$ to refer to the $k$th iterate of the natural logarithm.)

We may also put restrictions on our factorizations. Let $F(n)$ and $G(n)$ be the number of unordered and ordered factorizations of $n$ into pairwise coprime integers larger than one. Though Warlimont did not find asymptotic formulae for the sums of $F(n)$ and $G(n)$, he did establish upper and lower bounds \cite{Wa}. He proved that there exist positive constants $c_1, c_2, c_3, c_4$ such that
\[x\exp\left(c_1 \sqrt{\frac{\log x}{\log_2 x}}\right) \ll \sum_{n \leq x} F(n) \ll x\exp\left(c_2 \sqrt{\frac{\log x}{\log_2 x}}\right),\]
\[x\exp\left(c_3 \frac{\log x}{\log_2 x}\right) \ll \sum_{n \leq x} G(n) \ll x\exp\left(c_4 \frac{\log x}{\log_2 x}\right).\]

We refine Warlimont's estimates and obtain the following results.

\begin{thm} \label{F beta = 1} We have
\[\sum_{n \leq x} F(n) = x\exp\left((c + o(1))\sqrt{\frac{\log x}{\log_2 x}}\right),\]
where
\[c = 2\sqrt{2} e^{-\gamma/2} \approx 2.12,\]
with $\gamma \approx 0.577$ referring to the Euler-Mascheroni constant.
\end{thm}

In order to state the corresponding result for $G$, we must define two functions. Let $W$ be the Lambert $W$ function, i.e. the inverse of the function $h(z) = ze^z$. Define the exponential integral as
\[\textrm{Ei}(x) = -\int_{-x}^\infty t^{-1} e^{-t} dt.\]

\begin{thm} \label{G beta = 1} Let $w = W(1/(e \log 2))$. We have
\[\sum_{n \leq x} G(n) = x\exp\left((c + o(1)) \frac{\log x}{\log_2 x}\right),\]
where
\[c = w(1 + e^w \mathrm{Ei}(-w)(w + \log w + \log_2 2)) \approx 0.771.\]
\end{thm}

In addition, we consider the $\beta$-th moments of these functions for all $\beta$. Pollack recently estimated the positive moments of $f$. Let
\[L(x) = \exp(\log x \log_3 x/\log_2 x).\]

\begin{thm}[\cite{P}] \label{f beta > 1} For $\beta > 1$,
\[\sum_{n \leq x} f(n)^\beta = \frac{x^\beta}{L(x)^{\beta + o(1)}}.\]
\end{thm}

\begin{thm}[\cite{P}] \label{f beta < 1} For $\beta \in (0, 1)$,
\[\sum_{n \leq x} f(n)^\beta = x\exp\left((1 + o(1)) \left(\frac{(1 - \beta) \log_2 x}{(\log_3 x)^\beta}\right)^{1/(1 - \beta)}\right).\]
\end{thm}

We extend Pollack's results to $F$ and $G$. We write their positive moments, then their negative ones.

\begin{thm} \label{F beta > 1} For $\beta > 1$,
\[\sum_{n \leq x} F(n)^\beta = \frac{x^\beta}{L(x)^{2\beta + o(1)}}.\]
\end{thm}

\begin{thm} \label{F beta < 1} For $\beta \in (0, 1)$,
\[\sum_{n \leq x} F(n)^\beta = x\exp\left((1 + o(1)) \left(\frac{(1 - \beta) \log_2 x}{(\log_3 x)^\beta}\right)^{1/(1 - \beta)}\right).\]
\end{thm}

\begin{thm} \label{G beta > 1} For $\beta > 1$,
\[\sum_{n \leq x} G(n)^\beta = \frac{x^\beta}{L(x)^{\beta + o(1)}}.\]
\end{thm}

\begin{thm} \label{G beta < 1} For $\beta \in (0, 1)$,
\[\sum_{n \leq x} G(n)^\beta = x\exp\left((1 + o(1)) \frac{1 - \beta}{(\log 2)^{\beta/(1 - \beta)}} (\log_2 x)^{1/(1 - \beta)}\right).\]
\end{thm}

In the process of obtaining the negative moments of $F$ and $G$, we also obtain the negative moments of $f$ and $g$.

\begin{thm} \label{F beta < 0} For $\beta > 0$, the sums
\[\sum_{n \leq x} f(n)^{-\beta}, \quad \sum_{n \leq x} F(n)^{-\beta}\]
are both
\[\frac{x}{\log x} \exp((1 + o(1))((1 + \beta)(\log_2 x)(\log_3 x)^\beta)^{1/(1 + \beta)}).\]
\end{thm}

\begin{thm} \label{g beta < 0} For $\beta > 0$, the sums
\[\sum_{n \leq x} g(n)^{-\beta}, \quad \sum_{n \leq x} G(n)^{-\beta}\]
are both
\[\frac{x}{\log x} \exp((1 + o(1))(1 + \beta)(\log 2)^{\beta/(1 + \beta)} (\log_2 x)^{1/(1 + \beta)}).\]
\end{thm}

With the exception of the first moment, the positive moments of $g$ are still unknown.

\section{Bounds on $\pi(x,k)$}

Let $\omega(n)$ be the number of distinct prime factors of $n$. In addition, let $\pi(x, k)$ be the number of $n \leq x$ satisfying $\omega(n) = k$ and let $\pi'(x, k)$ have the additional restriction that each $n$ must be squarefree. Throughout this paper, we use several bounds for $\pi(x, k)$, each of which corresponds to a distinct set of possible values of $k$. We collect those bounds here.

The first result is the Hardy-Ramanujan Inequality.

\begin{thm}[{\cite{H}}] There exist constants $c_1$ and $c_2$ such that for all $x \geq 2$, $k \geq 1$,
\[\pi(x, k) \leq c_1 \frac{x}{\log x} \cdot \frac{(\log_2 x + c_2)^{k - 1}}{(k - 1)!}.\]
\end{thm}

Sathe and Selberg proved the following result.

\begin{thm}[{\cite{S, Se}}] \label{Sathe} Let $\pi'(x, k)$ be the number of squarefree $n \leq x$ satisfying $\omega(n) = k$. If $k \ll \log_2 x$, then
\[C_1 \frac{x}{\log x} \cdot \frac{(\log_2 x)^{k - 1}}{(k - 1)!} \leq \pi'(x, k) \leq C_2 \frac{x}{\log x} \cdot \frac{(\log_2 x)^{k - 1}}{(k - 1)!}\]
for fixed positive constants $C_1$ and $C_2$.
\end{thm}

Though the Hardy-Ramanujan Inequality is the most widely applicable formula for $\pi(x, k)$, it is also the least precise. Each of the following results applies to a smaller set of possible values of $k$, but gives a more precise estimate. From here on, let
\[L_0 = \log_2 x - \log k - \log_2 k.\]
The following bounds come from Pomerance, Hildebrand-Tenenbaum, and Kerner.

\begin{thm}[{\cite[Thms. $3.1$, $4.1$]{Pom}}] \label{Pomerance bound 2} For
\[(\log_2 x)^2 \leq k \leq \frac{\log x}{3 \log_2 x},\]
we have
\[\pi(x, k) = \frac{x}{k! \log x} \exp(k \log L_0 + o(k)).\]
\end{thm}

\begin{thm}[{\cite[Cor. $2$]{HiT}}] \label{Hildebrand Tenenbaum} For $(\log_2 x)^2 \ll k \ll (\log x)/(\log_2 x)^2$, we have
\[\pi(x, k) = \frac{x}{k! \log x} \exp\left(k\left(\log M + \frac{1}{M} + O(R)\right)\right),\]
with
\[M = \log \xi + \log_2 \xi - \log L_0 - \gamma,\]
\[R = \frac{1}{L_0} \left(\frac{1}{y} + \frac{1}{L_0}\right),\]
\[y = \frac{k}{L_0},\]
\[\xi = \frac{\log x}{y \log y}.\]
\end{thm}

\begin{thm}[{\cite[Thm. $6.1$]{Pom}}] \label{Pomerance bound} For $k = \lfloor c \log x/\log_2 x \rfloor$ with $c \in [1/3, 1 - (1/\log_2 x)]$, we have
\[\pi(x, k) = xk^{-k} (1 - c)^k e^{O(k)}.\]
If $c > 1 - (1/\log_2 x)$, then
\[\pi(x, k) = e^{O(k)}.\]
\end{thm}

\begin{thm}[{\cite[Section $1.2.3$, Cor. $1$]{Ke}}] Fix $c \in (\epsilon, 1 - \epsilon)$. Let $H(t)$ be the inverse function of $te^t \Gamma(0, t)$. Define
\[S(t) = \frac{t}{\pi} \sin\left(\frac{\pi}{t}\right).\]
Let
\[\alpha = 1 + (H(c) + o(1))\frac{1}{\log_2 x}.\]
For $k = c \log x/\log_2 x$, we have
\[\pi(x, k) = x^\alpha \left(\frac{e}{S(\alpha) \log x}\right)^{k\alpha} \exp\left(O\left(\frac{k}{(\log_2 x)^{1/4}}\right)\right).\]
\end{thm}

\section{Positive moments of $F$}

Let $n = p_1^{e_1} \cdots p_k^{e_k}$. Suppose we write $n$ as a product $a_1 \cdots a_\ell$, where $(a_i, a_j) = 1$ for all $i \neq j$. For each $i \leq k$, there exists a unique $j \leq \ell$ such that $p_i^{e_i} \| a_j$. So, the values of the $e_i$'s are irrelevant to $F(n)$. Indeed, $F(n)$ is the simply the number of set partitions of a $k$ element set. The number of such partitions is $B_k$, the $k$th Bell number. In general, we have
\[F(n) = B_{\omega(n)}.\]

Using this equation, we may rewrite the sum of $F(n)^\beta$. For any $k$, let $\pi(x, k)$ be the number of $n \leq k$ satisfying $\omega(n) = k$. We have
\[\sum_{n \leq x} F(n)^\beta = \sum_{n \leq x} B_{\omega(n)}^\beta = \sum_k B_k^\beta \pi(x, k).\]
A simple lower bound for this sum is
\[\max_k B_{k}^\beta \pi(x, k).\]
By the Prime Number Theorem, the maximum value of $\omega(n)$ for any $n \leq x$ is
\[\frac{\log x}{\log_2 x}\left(1 + O\left(\frac{1}{\log_2 x}\right)\right).\]
This result allows us to put an upper bound on our sum:
\[\sum_{n \leq x} F(n)^\beta \leq \sum_{k \leq 2\frac{\log x}{\log_2 x}} B_k^\beta \pi(x, k) \leq \sum_{k \leq 2\frac{\log x}{\log_2 x}} \max_k B_k^\beta \pi(x, k) \ll \frac{\log x}{\log_2 x} \max_k B_k^\beta \pi(x, k).\]
Putting this together gives us
\[\max_k B_k^\beta \pi(x, k) \ll \sum_{n \leq x} F(n)^\beta \ll \frac{\log x}{\log_2 x} \max_k B_k^\beta \pi(x, k),\]
hence
\[\sum_{n \leq x} F(n)^\beta = \exp(O(\log_2 x)) \max_k B_k^\beta \pi(x, k).\]

Note that in Theorems \ref{F beta = 1}, \ref{F beta > 1}, and \ref{F beta < 1}, $\exp(O(\log_2 x))$ is smaller than the error terms we are trying to obtain, rendering the $\exp(O(\log_2 x))$ irrelevant. In order to estimate the sum of $F(n)^\beta$, we simply need to determine the maximum value of $B_k^\beta \pi(x, k)$.

Before moving further, we provide a short proof of Theorem \ref{F beta < 1}. It is clear that the right-hand side is an upper bound for the left because $F(n) \leq f(n)$ and we already have the sum of $f(n)^\beta$ from Theorem \ref{f beta < 1}. As to showing it is a lower bound, we note that Pollack actually proved Theorem \ref{f beta < 1} by demonstrating that
\[\max_k B_k^\beta \pi(x, k) = x\exp\left((1 + o(1))\left(\frac{(1 - \beta) \log_2 x}{(\log_3 x)^\beta}\right)^{1/(1 - \beta)}\right).\]

Throughout this section and Section $5$, we use the following formula for $B_k$ to varying degrees of precision.

\begin{thm}[{\cite[Eq. (6.27)]{dB}}] We have
\[B_k = \exp\left(k \log k - k \log_2 k - k + \frac{k \log_2 k}{\log k} + \frac{k}{\log k} + O\left(\frac{k(\log_2 k)^2}{(\log k)^2}\right)\right).\]
\end{thm}

\subsection{The $\beta > 1$ case}

We prove Theorem \ref{F beta > 1}, which we restate below. Our proof is very similar to the proof of Theorem \ref{f beta > 1}.

\begin{customthm}{1.5} For $\beta > 1$,
\[\sum_{n \leq x} F(n)^\beta = \frac{x^\beta}{L(x)^{2\beta + o(1)}}.\]
\end{customthm}

\begin{proof} The lower bound comes from a straightforward argument. We note that
\[\sum_{n \leq x} F(n)^\beta \geq \max_{n \leq x} F(n)^\beta.\]
The maximum value of $F(n)$ is $B_k$, where $k$ is the largest possible value of $\omega(n)$ for any $n \leq x$. We wrote earlier that this value of $k$ is
\[\frac{\log x}{\log_2 x} + O\left(\frac{\log x}{(\log_2 x)^2}\right).\]
In this case,
\[\log k = \log_2 x - (1 + o(1)) \log_3 x\]
\[\log_2 k = (1 + o(1)) \log_3 x,\]
and
\[B_k = \exp(k \log k + O(k \log_2 k)) = \exp\left(\log x - (2 + o(1)) \frac{\log x \log_3 x}{\log_2 x}\right) = \frac{x}{L(x)^{2 + o(1)}}.\]
Therefore,
\[\sum_{n \leq x} F(n)^\beta \geq \frac{x^\beta}{L(x)^{2\beta + o(1)}}.\]

We now establish the right-hand side as an upper bound.

Let $\mathcal{S}_1$ be the set of $n \leq x$ satisfying $F(n) \leq x/L(x)^{2\beta/(\beta - 1)}$. We have
\[\sum_{n \in \mathcal{S}_1} F(n)^\beta \leq \left(\max_{n \in\mathcal{S}_1} F(n)\right)^{\beta - 1} \sum_{n \leq x} F(n) \leq \frac{x^{\beta - 1}}{L(x)^{2\beta + o(1)}} \cdot xL(x)^{o(1)} = \frac{x^\beta}{L(x)^{2\beta + o(1)}}.\]

Let $\mathcal{S}_2$ be all other $n \leq x$. We show that
\[k > \frac{\log x}{\log_2 x} \left(1 - C \frac{\log_3 x}{\log_2 x}\right)\]
for a positive constant $C$.

By assumption,
\[B_k = F(n) > \frac{x}{L(x)^{2\beta/(\beta - 1)}}.\]
Recall that
\[B_k = \exp(k \log k - (1 + o(1)) k \log_2 k).\]
Therefore,
\[k \log k - (1 + o(1)) k \log_2 k > \log\left(\frac{x}{L(x)^{2\beta/(\beta - 1)}}\right) = \log x - \frac{2\beta}{\beta - 1} \left(\frac{\log x \log_3 x}{\log_2 x}\right).\]
Suppose $k = (1 - \delta)(\log x/\log_2 x)$ for some $\delta = o(1)$ as $x \to \infty$. We have
\[k \log k = (1 - \delta) \frac{\log x}{\log_2 x} \cdot (\log_2 x - (1 + o(1)) \log_3 x) = (1 - \delta) \log x - (1 + o(1)) \frac{\log x \log_3 x}{\log_2 x},\]
\[k \log_2 k = (1 - \delta) \frac{\log x}{\log_2 x} \cdot (1 + o(1)) \log_3 x = (1 + o(1)) \frac{\log x \log_3 x}{\log_2 x}.\]
Hence,
\[k \log k - (1 + o(1)) k \log_2 k = (1 - \delta) \log x - (2 + o(1)) \frac{\log x \log_3 x}{\log_2 x}.\]

If
\[\delta < \left(\frac{2}{\beta - 1} - \epsilon\right) \cdot \frac{\log_3 x}{\log_2 x}\]
for a fixed $\epsilon > 0$, then
\[k \log k - (1 + o(1)) k \log_2 k > \log x - \frac{2\beta}{\beta - 1} \left(\frac{\log x \log_3 x}{\log_2 x}\right).\]
We have
\[k > \frac{\log x}{\log_2 x} \left(1 - C \frac{\log_3 x}{\log_2 x}\right)\]
for a positive constant $C$.

Because $n \in \mathcal{S}_2$, we know that $\omega(n) = (1 + o(1))(\log x/\log_2 x)$. Under this assumption, we only need to evaluate
\[\max_{k > \left(1 - C \frac{\log_3 x}{\log_2 x}\right) \frac{\log x}{\log_2 x}} B_k^\beta \pi(x, k).\]

We now prove that $\pi(x, k) = L(x)^{o(1)}$. Suppose $k > (\log x/\log_2 x)(1 - (1/\log_2 x))$. By Theorem \ref{Pomerance bound},
\[\pi(x, k) = e^{O(k)} = e^{O(\log x/\log_2 x)} = L(x)^{o(1)}.\]
In addition, $B_k^\beta \leq x^{2\beta}/L(x)^{2\beta + o(1)}$. Therefore, $B_k^\beta \pi(x, k) = x^{2\beta}/L(x)^{2\beta + o(1)}$ as well.

Suppose $k \leq (\log x/\log_2 x)(1 - (1/\log_2 x))$. Specifically, we write $k = c(\log x/\log_2 x)$ with $1 - (C \log_3 x/\log_2 x) < c \leq 1 - (1/\log_2 x)$. Reapplying Theorem \ref{Pomerance bound} gives us
\begin{eqnarray*}
B_k^\beta \pi(x, k) & = & xe^{O(k)} \exp(\beta k \log k - (\beta + o(1)) k \log_2 k  - k \log k + k\log(1 - c)) \\
& = & xL(x)^{o(1)} \exp((\beta - 1) k \log k - (1 + o(1)) \beta k \log_2 k + k\log(1 - c)).
\end{eqnarray*}
Note that $k = (1 + o(1)) \log x/\log_2x$ and $k < \log x/\log_2 x$. We bound each term in the exponential separately:
\[\exp((\beta - 1)k \log k) \ll \exp\left((\beta - 1)\frac{\log x}{\log_2 x} (\log_2 x - \log_3 x)\right) = \frac{x^{\beta - 1}}{L(x)^{\beta - 1}},\]
\[\exp((1 + o(1))\beta k \log_2 k) \gg \exp\left((1 + o(1)) \beta \frac{\log x}{\log_2 x} \log_3 x\right) = L(x)^{\beta + o(1)}.\]
Finally,
\[\log(1 - c) < \log\left(\frac{C \log_3 x}{\log_2 x}\right) = -(1 + o(1)) \log_3 x,\]
giving us
\[\exp(k \log(1 - c)) < \exp\left(-(1 + o(1)) \frac{\log x}{\log_2 x} \cdot \log_3 x\right) = L(x)^{-1 + o(1)}.\]
Putting everything together gives us
\[\max_{k > \left(1 - C \frac{\log_3 x}{\log_2 x}\right) \frac{\log x}{\log_2 x}} B_k^\beta \pi(x, k) \ll \frac{x^\beta}{L(x)^{2\beta + o(1)}},\]
which implies that
\[\sum_{n \in \mathcal{S}_2} F(n)^\beta \ll \frac{x^\beta}{L(x)^{2\beta + o(1)}}.\]
\end{proof}

\subsection{The $\beta = 1$ case}

First we show that the sum of $F(n)$ for $n$ having many prime factors is negligible. Before doing so, we record another estimate for $\pi(x, k)$ for certain values of $k$.

\begin{lemma} For all $k \gg (\log x)/(\log_2 x)^2$,
\[B_k \pi(x, k) \ll x.\]
\end{lemma}

\begin{proof} We split the proof into three cases:
\begin{enumerate}
\item{\[k \geq \frac{\log x}{\log_2 x} \left(1 - \frac{1}{\log_2 x}\right),\]}
\item{\[\frac{\log x}{3\log_2 x} < k < \frac{\log x}{\log_2 x} \left(1 - \frac{1}{\log_2 x}\right),\]}
\item{\[\frac{\log x}{(\log_2 x)^2} \ll k \leq \frac{\log x}{3\log_2 x},\]}
\end{enumerate}

We show each of these cases has the correct upper bound.

Suppose $k \geq (\log x/\log_2 x)(1 - (1/\log_2 x))$. By Theorem \ref{Pomerance bound},
\[\pi(x, k) = \exp(O(k)).\]
We have
\begin{eqnarray*}
B_k \pi(x, k) & = & \exp(k \log k - (1 + o(1))(k \log_2 k)) \\
& \ll & \exp(k \log k) \\
& \ll & \exp\left(\frac{\log x}{\log_2 x} \left(1 + O\left(\frac{1}{\log_2 x}\right)\right) (\log_2 x - (1 + o(1))\log_3 x)\right) \\
& = & \frac{x}{L(x)^{1 + o(1)}}.
\end{eqnarray*}

Suppose $k = \lfloor c \log x/\log_2 x \rfloor$ with
\[\frac{1}{3} < c < 1 - \frac{1}{\log_2 x}.\]
By Theorem \ref{Pomerance bound},
\[\pi(x, k) = xk^{-k} (1 - c)^k e^{O(k)} \ll xk^{-k} e^{O(k)} = x\exp(-k \log k + O(k)).\]
Therefore,
\[B_k \pi(x, k) \ll x\exp(-(1 + o(1)) k \log_2 k) \ll x.\]

Suppose
\[\frac{\log x}{(\log_2 x)^2} \ll k \leq \frac{\log x}{3 \log_2 x}.\]
By Theorem \ref{Pomerance bound 2},
\[\pi(x, k) = \frac{x}{k! \log x} \exp(k \log L_0 + o(k)),\]
with $L_0 = \log_2 x - \log k - \log_2 k$. Recall that
\[B_k = \exp(k \log k - k \log_2 k + O(k)).\]
Thus,
\[B_k \pi(x, k) = x\exp(k \log L_0 - k \log_2 k + O(k)).\]
Note that $\log L_0 - \log_2 k$ decreases as $k$ increases. For $k = \log x/(\log_2 x)^2$,
\[L_0 = \log_2 x - (\log_2 x - 2\log_3 x) - (1 + o(1)) \log_3 x = (1 + o(1)) \log_3 x,\]
\[\log L_0 = O(\log_4 x),\]
\[\log_2 k = (1 + o(1)) \log_3 x.\]
Therefore,
\[B_k \pi(x, k) \ll x\exp(-(\log_3 x + o(1)) k) \ll x.\]
\end{proof}

Before proving the main result, we state a few preliminary results.

\begin{lemma} We have
\[W(k) = \log k - \log_2 k + o(1).\]
\end{lemma}

\begin{lemma}[{\cite{MW}}] \label{Bell lemma} For $h = O(1)$, we have
\[B_{k + h} = \frac{(k + h)!}{W(k)^{k + h}} \cdot \frac{\exp(e^{W(k)} - 1)}{(2\pi(W(k)^2 + W(k)) e^{W(k)})^{1/2}} \cdot (1 + O(e^{-W(k)})).\]
\end{lemma}

\begin{cor} We have
\[\frac{B_{k + 1}}{B_k} = \frac{k + 1}{W(k)} (1 + O(e^{-W(k)})).\]
\end{cor}

\begin{cor} For all $k \ll \log x/(\log_2 x)^2$,
\[\frac{B_{k + 1} \pi(x, k + 1)}{B_k \pi(x, k)} = \frac{L_0}{W(k)} \left(1 + O\left(e^{-W(k)} + \frac{1}{k} + \frac{\log L_0}{L_0}\right)\right).\]
\end{cor}

\begin{proof} By \cite[Cor. $3$]{HiT},
\[\frac{\pi(x, k + 1)}{\pi(x, k)} = \frac{L_0}{k} \left(1 + O\left(\frac{\log L_0}{L_0}\right)\right)\]
for $k \ll (\log x)/(\log_2 x)^2$. Combining this with the previous corollary gives us the desired result.
\end{proof}

\begin{customthm}{1.1} We have
\[\sum_{n \leq x} F(n) = x\exp\left((c + o(1))\sqrt{\frac{\log x}{\log_2 x}}\right),\]
where
\[c = 2\sqrt{2} e^{-\gamma/2},\]
with $\gamma$ being the Euler-Mascheroni constant.
\end{customthm}

\begin{proof} Because of Lemma $3.3$, we may assume that $\omega(n) \ll \log x/(\log_2 x)^2$. We maximize $B_k \pi(x, k)$ by selecting an optimal value of $k$.

We use the previous corollary to show that the maximum value of $B_k \pi(x, k)$ occurs at $k = (\log x)^{1/2} (\log_2 x)^{C + o(1)}$ for some constant $C$. For an optimal value of $k$,
\[\frac{B_{k + 1} \pi(x, k + 1)}{B_k \pi(x, k)} \leq 1 \leq \frac{B_k \pi(x, k)}{B_{k - 1} \pi(x, k - 1)},\]
which occurs when $L_0 \sim W(k) \sim \log k$. Recall that
\[L_0 = \log_2 x - \log k - \log_2 k = \log_2 x - (1 + o(1)) \log k.\]
Solving
\[\log_2 x - (1 + o(1)) \log k = (1 + o(1)) \log k\]
gives us
\[k = (\log x)^{(1/2) + o(1)}.\]

We may refine this estimate further. Let $k = (\log x)^{1/2} (\log_2 x)^{C + o(1/\log_3 x)}$ with $C = o(\log_2 x/\log_3 x)$. We now show that $C = O(1)$, then use this result to bound $C$ more precisely. We have
\[L_0 = \log_2 x - \left(\frac{1}{2} \log_2 x + C \log_3 x\right) - (1 + o(1)) \log_3 x = \frac{1}{2} \log_2 x - (C + 1 + o(1)) \log_3 x,\]
\[W(k) = \left(\frac{1}{2} \log_2 x + C \log_3 x\right) - (1 + o(1)) \log_3 x = \frac{1}{2} \log_2 + (C - 1 + o(1)) \log_3 x.\]
Hence,
\[\frac{L_0}{W(k)} = 1 - (4C + o(1)) \frac{\log_3 x}{\log_2 x}.\]
Note that
\[\frac{L_0}{W(k)} = 1 + O\left(e^{-W(k)} + \frac{1}{k} + \frac{\log L_0}{L_0}\right) = 1 + O\left(\frac{\log_3 x}{\log_2 x}\right).\]
Therefore, the optimal value of $C$ is bounded. We may assume that $k = (\log x)^{1/2} (\log_2 x)^C$ with $C = O(1)$.

Using Theorem \ref{Hildebrand Tenenbaum}, we prove that if
\[k = (\log x)^{1/2} (\log_2 x)^C,\]
then
\[B_k \pi(x, k)\]
\[= x\exp((2 - 4C)(\log x)^{1/2} (\log_2 x)^{C - 1} \log_3 x + 2(2 -\gamma - \log 2 + o(1)) (\log x)^{1/2} (\log_2 x)^{C - 1}).\]
The maximum value of this function occurs at $C = 1/2$, at which point we obtain the desired formula.

We find the values of the variables used in Theorem \ref{Hildebrand Tenenbaum}:
\[L_0 = \log_2 x - \log k - \log_2 k = \log_2 x - \left(\frac{1}{2} + o(1)\right) \log_2 x - o(\log_2 x) \sim \frac{1}{2} \log_2 x,\]
\[y = \frac{k}{L_0} \sim 2(\log x)^{1/2} (\log_2 x)^{C - 1},\]
\[\xi = \frac{\log x}{y \log y} \sim \frac{\log x}{2(\log x)^{1/2} (\log_2 x)^{C - 1} \cdot \frac{1}{2} \log_2 x} \sim \frac{(\log x)^{1/2}}{(\log_2 x)^C},\]
\begin{eqnarray*}
M & = & \log \xi + \log_2 \xi - \log L_0 - \gamma \\
& = & \left(\frac{1}{2} \log_2 x - C \log_3 x + o(1)\right) + (\log_3 x - \log 2 + o(1)) \\
& & - (\log_3 x - \log 2 + o(1)) - \gamma \\
& = & \frac{1}{2} \log_2 x - C \log_3 x - \gamma + o(1),
\end{eqnarray*}
\[\log M = \log_3 x - \log 2 - \frac{2C \log_3 x}{\log_2 x} - (1 + o(1))\frac{2\gamma}{\log_2 x},\]
\[R = \frac{1}{L_0} \left(\frac{1}{\log y} + \frac{1}{L_0}\right) = O\left(\frac{1}{(\log_2 x)^2}\right).\]

We now apply Theorem \ref{Hildebrand Tenenbaum}:
\begin{eqnarray*}
\pi(x, k) & = & \frac{x}{k! \log x} \exp\left(k \log M + \frac{k}{M} + O(kR)\right) \\
& = & x \exp\left(k \log M + \frac{k}{M} - k \log k + k + O\left(\frac{k}{(\log_2 x)^2}\right)\right).
\end{eqnarray*}

Recall that
\[B_k = \exp\left(k \log k - k \log_2 k - k + \frac{k \log_2 k}{\log k} + (1 + o(1)) \frac{k}{\log k}\right).\]
Therefore,
\[B_k \pi(x, k) = x\exp\left(k \log M - k \log_2 k + \frac{k}{M} + \frac{k \log_2 k}{\log k} + (1 + o(1)) \frac{k}{\log k}\right).\]
Note that
\begin{eqnarray*}
\log M - \log_2 k & = & \left(\log_3 x - \log 2 - \frac{2C \log_3 x}{\log_2 x} - \frac{2\gamma + o(1)}{\log_2 x}\right) \\
& & - \left(\log_3 x - \log 2 + \frac{2C \log_3 x}{\log_2 x}\right) \\
& = & -\frac{4C \log_3 x + 2\gamma + o(1)}{\log_2 x}.
\end{eqnarray*}

In addition,
\[\frac{1}{M} \sim \frac{2}{\log_2 x},\]
\[\frac{1}{\log k} \sim \frac{2}{\log_2 x},\]
\[\frac{\log_2 k}{\log k} = \left(\frac{1}{2} \log_2 x + C \log_3 x\right)^{-1} (\log_3 x - \log 2 + o(1)) = \frac{2\log_3 x - 2 \log 2 + o(1)}{\log_2 x}.\]

Putting everything together gives us
\[B_k \pi(x, k) = x\exp\left(k\left(\frac{(2 - 4C)\log_3 x + (4 - 2\gamma - 2\log 2 + o(1))}{\log_2 x}\right)\right).\]
Hence,
\[B_k \pi(x, k)\] 
\[ = x\exp((2 - 4C)(\log x)^{1/2} (\log_2 x)^{C - 1} (\log_3 x) + 2(2 - \gamma - \log 2 + o(1)) (\log x)^{1/2} (\log_2 x)^{C - 1}).\]

We choose $C$ to maximize $B_k \pi(x, k)$. Specifically, we factor out various terms which are independent of $C$ and maximize
\[(\log_2 x)^C ((1 - 2C)(\log_3 x) + 2 - \gamma - \log 2).\]
Its derivative with respect to $C$ is
\[(\log_2 x)^C (\log_3 x)((1 - 2C)(\log_3 x) - (\gamma + \log 2)).\]
Setting this quantity equal to $0$ gives us
\[C = \frac{1}{2} - \left(\frac{\gamma + \log 2}{2}\right) \frac{1}{\log_3 x}.\]
From this, we obtain
\[k = (\log x)^{1/2} (\log_2 x)^{C + o(1/\log_3 x)} \sim \sqrt{(1/2e^\gamma) \log x \log_2 x}.\]
Plugging this value of $k$ into our formula for $B_k \pi(x, k)$ gives us our desired result.
\end{proof}

We note here that Warlimont obtained his lower bound by setting $k \sim \sqrt{\log x \log_2 x}$, whereas we have used $k \sim \sqrt{(1/2e^\gamma) \log x \log_2 x}$.

\section{Positive moments of $G$}

The techniques we used for $F$ also hold for $G$. Here, we use the ordered Bell numbers, which we denote $a_k$, instead of the unordered Bell numbers. We now have $G(n) = a_{\omega(n)}$. Thus,
\[\sum_{n \leq x} G(n)^\beta = \sum_k a_k^\beta \pi(x, k) = e^{O(\log_2 x)} \max_k a_k^\beta \pi(x, k).\]
Once again, the $e^{O(\log_2 x)}$ term is negligible, allowing us to focus entirely on maximizing $a_k^\beta \pi(x, k)$.

\begin{thm}[{\cite{Sk}}] We have
\[a_k \sim \frac{1}{2 \log 2} \cdot \frac{k!}{(\log 2)^k} = \exp(k \log k - (1 + \log_2 2)k + O(\log k)).\]
\end{thm}

\subsection{The $\beta > 1$ case}

In the previous section, we established that if $\beta > 1$, then
\[\sum_{n \leq x} F(n)^\beta = \frac{x^\beta}{L(x)^{2\beta + o(1)}}.\]
Already having this bound allows us to write a short proof for the sum of $G(n)^\beta$.

\begin{customthm}{1.7} For $\beta > 1$,
\[\sum_{n \leq x} G(n)^\beta = \frac{x^\beta}{L(x)^{\beta + o(1)}}.\]
\end{customthm}

\begin{proof} First, we establish the right-hand side as a lower bound.
\[\sum_{n \leq x} G(n)^\beta \geq \max_{n \leq x} G(n)^\beta = \max_{n \leq x} a_{\omega(n)}^\beta.\]
Recall that the maximum value of $k = \omega(n)$ is $(1 + O(1/\log_2 x)) \log x/\log_2 x$. For this $k$,
\[a_k = \exp(k \log k + O(k)) = \exp\left(\frac{\log x}{\log_2 x} (\log_2 x - \log_3 x) + O\left(\frac{\log x}{\log_2 x}\right)\right) = \frac{x}{L(x)^{1 + o(1)}}.\]
Therefore,
\[\sum_{n \leq x} G(n)^\beta \geq \frac{x^\beta}{L(x)^{\beta + o(1)}}.\]

For the upper bound, we want to find the maximum value of $a_k^\beta \pi(x, k)$. Note that
\[\max_k a_k^\beta \pi(x, k) \leq \max_k \left(\frac{a_k}{B_k}\right)^\beta B_k^\beta \pi(x, k) \leq \left(\max_k \frac{a_k}{B_k}\right)^\beta \max_k B_k^\beta \pi(x, k).\]
However,
\[\max_k B_k^\beta \pi(x, k) = L(x)^{o(1)} \sum_{n \leq x} F(n)^\beta = \frac{x^\beta}{L(x)^{2\beta + o(1)}}.\]

Recall that
\[a_k = \exp(k \log k + o(k \log_2 k)),\]
\[B_k = \exp(k \log k - (1 + o(1)) k \log_2 k).\]
Therefore,
\[a_k/B_k = \exp((1 + o(1)) k\log_2 k),\]
which implies that
\[a_k/B_k \leq L(x)^{1 + o(1)}\]
for all $k \leq (1 + o(1)) \log x/\log_2 x$.

Putting all this together gives us
\[\sum_{n \leq x} G(n)^\beta = \frac{x^\beta}{L(x)^{\beta + o(1)}}.\]
\end{proof}

\subsection{The $\beta < 1$ case}

We prove Theorem \ref{G beta < 1}, which is the sum of $G(n)^\beta$ for $\beta \in (0, 1)$.

\begin{customthm}{1.8} For all $\beta \in (0, 1)$,
\[\sum_{n \leq x} G(n)^\beta = x\exp\left((1 + o(1)) \frac{1 - \beta}{(\log 2)^{\beta/(1 - \beta)}} (\log_2 x)^{1/(1 - \beta)}\right).\]
\end{customthm}

\begin{proof} Our goal is to maximize $a_k^\beta \pi(x, k)$ for $k \leq (1 + o(1))(\log x/\log_2 x)$. By the Hardy-Ramanujan Inequality,
\[\pi(x, k) \ll \frac{x}{\log x} \left(\frac{(\log_2 x)^k}{(k - 1)!}\right) = x\exp(k \log_3 x - k \log k + k + O(\log_2 x)).\]
In addition,
\[a_k^\beta = \exp(\beta k \log k - \beta(1 + \log_2 2) k + O(1)).\]
Therefore,
\[a_k^\beta \pi(x, k) \ll x\exp(k \log_3 x - (1 - \beta)k \log k + (1 - \beta - \beta \log_2 2)k + O(\log_2 x)).\]

If $k \ll (\log_2 x)^{(1/(1 - \beta)) - \epsilon}$ for some $\epsilon > 0$, then
\[a_k^\beta \pi(x, k) \ll x\exp(k \log_3 x) \ll x\exp((\log_2 x)^{(1/(1 - \beta)) - \epsilon + o(1)}) = x\exp(o((\log_2 x)^{1/(1 - \beta)})),\]
which is contained in the error term of our desired formula. If $k \gg (\log_2 x)^{(1/(1 - \beta)) + \epsilon}$ for some $\epsilon$, then
\[\log_3 x - (1 - \beta)\log k \leq -(1 - \beta)\epsilon \log_3 x + O(1),\]
which implies that
\[a_k^\beta \pi(x, k) = o(x).\]

We may assume that $k = (\log_2 x)^{(1/(1 - \beta)) + o(1)}$. We write $k = C(\log_2 x)^{1/(1 - \beta)}$ with $C = (\log_2 x)^{o(1)}$ and optimize $C$. Note that we currently have
\[a_k^\beta \pi(x, k) \ll x\exp(C((1 - \beta)(1 - \log C) - \beta\log_2 2)(\log_2 x)^{1/(1 - \beta)} + O(\log_2 x)).\]
We maximize $C((1 - \beta)(1 - \log C) - \beta \log_2 2)$. As $C$ approaches $0$, this quantity approaches $0$ and as $C$ approaches $\infty$, the quantity approaches $-\infty$. At the maximum,
\[\frac{d}{dC} (C((1 - \beta)(1 - \log C) - \beta \log_2 2)) = -(1 - \beta)(\log C) - \beta \log_2 2 = 0.\]
Therefore,
\[C = e^{-\beta(\log_2 2)/(1 - \beta)}\]
and
\[k = (1 + o(1))(\log 2)^{-\beta/(1 - \beta)} (\log_2 x)^{1/(1 - \beta)}.\]
We have
\[a_k^\beta \pi(x, k) \ll x\exp\left((1 + o(1))(\log 2)^{-\beta/(1 - \beta)} (1 - \beta)(\log_2 x)^{1/(1 - \beta)}\right).\]

We now establish the right-hand side as a lower bound. By Theorem \ref{Pomerance bound 2},
\[\pi(x, k) = \frac{x}{k! \log x} \exp(k\log L + o(k)),\]
where
\[L_0 = \log_2 x - \log k - \log_2 k\]
In this case,
\[L_0 = \log_2 x - \frac{1}{1 - \beta} (\log_3 x) - (1 + o(1)) \log_4 x,\]
\[\log L_0 = \log_3 x + o(1).\]
So,
\[\pi(x, k) = x\exp(k \log_3 x - k \log k + k + O(\log_2 x)),\]
making our upper and lower bounds equal.
\end{proof}

\subsection{The $\beta = 1$ case}

Suppose $\beta = 1$. Warlimont proved that there exist constants $c_1$ and $c_2$ such that
\[x\exp\left(c_1 \frac{\log x}{\log_2 x}\right) \ll \sum_{n \leq x} G(n) \ll x\exp\left(c_2 \frac{\log x}{\log_2 x}\right).\]
We obtain an asymptotic for the sum by summing over $a_k \pi(x, k)$ for all $k$. First, we bound the sum for
\[k < \epsilon \frac{\log x}{\log_2 x}, \quad k > (1 - \epsilon) \frac{\log x}{\log_2 x}\]
for a small value of $\epsilon$. Then, we bound it for
\[\epsilon \frac{\log x}{\log_2 x} < k < (1 - \epsilon) \frac{\log x}{\log_2 x}.\]

\begin{lemma} As $\epsilon \to 0$, we have
\[a_k \pi(x, k) = x\exp\left(o\left(\frac{\log x}{\log_2 x}\right)\right)\]
for
\[k \leq \epsilon \frac{\log x}{\log_2 x}.\] 
\end{lemma}

\begin{proof} Suppose $k < (\log_2 x)^2$. By Theorem $2.1$,
\[\pi(x, k) \ll \frac{x}{\log x} \cdot \frac{(\log_2 x)^{k - 1}}{(k - 1)!}.\]
In addition,
\[a_k \ll \frac{k!}{(\log 2)^k}.\]
Therefore,
\begin{eqnarray*}
a_k \pi(x, k) & \ll & \frac{x}{\log x} \cdot \frac{k(\log_2 x)^{k - 1}}{(\log 2)^k} \\
& \ll & x \exp((1 + o(1))k \log_3 x - \log_2 x) \\
& \ll & x\exp(O((\log_2 x)^2 \log_3 x)).
\end{eqnarray*}

Suppose $(\log_2 x)^2 \leq k \leq \epsilon(\log x/\log_2 x)$. By Theorem $2.3$,
\[\pi(x, k) = \frac{x}{k! \log x} \exp(k \log L_0 + o(k))\]
with $L_0 = \log_2 x - \log k - \log_2 k$. Therefore,
\[a_k \pi(x, k) = \frac{x}{(\log 2)^k} \exp(k \log L_0 + o(k)) \ll x\exp(k \log L_0 + o(k)).\]
We now maximize $k \log L_0$:
\[\frac{d}{dk} (k \log L_0) = \log L_0 + \frac{k}{L_0} \left(-\frac{1}{k} - \frac{1}{k \log k}\right) = \log L_0 - \frac{1}{L_0} - \frac{1}{L_0 \log k} > \log L_0 - \frac{2}{L_0}.\]
Note that $L_0$ decreases as $k$ increases. The minimum value of $\log L_0 - (2/L_0)$ occurs when $k$ is as large as possible. In this case,
\begin{eqnarray*}
L_0 & = & \log_2 x - \log k - \log_2 k \\
& = & \log_2 x - (\log_2 x - \log_3 x + \log \epsilon) - \log(\log_2 x - \log_3 x + \log \epsilon) \\
& \geq & \log_2 x - (\log_2 x - \log_3 x + \log \epsilon) - \log_3 x \\
& = & -\log \epsilon.
\end{eqnarray*}
Therefore, for sufficiently small $\epsilon$, we have $\log L_0 - (2/L_0) > 0$. The derivative of $k \log L_0$ with respect to $k$ is always positive. Because the derivative is always positive, we maximize $k \log L_0$ by setting $k$ to $\epsilon (\log x/\log_2 x)$. Therefore,
\[a_k \pi(x, k) \ll x\exp\left((\epsilon \log_2 (1/\epsilon) + o(1)) \frac{\log x}{\log_2 x}\right).\]
As $\epsilon$ approaches $0$, $\epsilon \log_2 (1/\epsilon)$ approaches $0$ as well.
\end{proof}

\begin{lemma} As $\epsilon \to 0$, we have
\[a_k \pi(x, k) = x\exp\left(o\left(\frac{\log x}{\log_2 x}\right)\right)\]
for
\[k \geq (1 - \epsilon) \frac{\log x}{\log_2 x}.\] 
\end{lemma}

\begin{proof} Suppose $k = c(\log x/\log_2 x)$ with $1 - \epsilon \leq c \leq 1 - (1/\log_2 x)$. By Theorem $2.4$,
\[\pi(x, k) = xk^{-k} (1 - c)^k e^{O(k)} \leq xk^{-k} e^{O(k)}.\]
There exists a constant $C$ such that
\[\pi(x, k) \ll x\exp(-k \log k + Ck).\]
In addition,
\[a_k = \exp(k \log k - (1 + \log_2 2 + o(1)) k).\]
Hence,
\[a_k \pi(x, k) \ll x\exp((C - 1 - \log_2 2 + o(1)) k) = x\exp\left((1 + o(1))(C - 1 - \log_2 2)\epsilon \frac{\log x}{\log_2 x}\right).\]
As $\epsilon$ goes to $0$, the coefficient of $\log x/\log_2 x$ goes to $0$ as well.

Suppose $c > 1 - (1/\log_2 x)$. Using Theorem $2.4$ again, we have $\pi(x, k) = e^{O(k)}$. We have
\[a_k \pi(x, k) \leq \exp(k \log k + O(k)) \leq \exp\left(\frac{\log x}{\log_2 x} \log\left(\frac{\log x}{\log_2 x}\right) + O\left(\frac{\log x}{\log_2 x}\right)\right) \ll x.\]
\end{proof}

We now prove the main result. Recall that $W$ is the inverse of the function $h(z) = ze^z$ and that
\[\textrm{Ei}(x) = -\int_{-x}^\infty t^{-1} e^{-t} dt.\]

\begin{customthm}{1.2} Letting $w = W(1/(e \log 2))$, we have
\[\sum_{n \leq x} G(n) = x\exp\left((c + o(1))\frac{\log x}{\log_2 x}\right)\]
with
\[c = w(1 + e^w \mathrm{Ei}(-w)(w + \log w + \log_2 2)).\]
\end{customthm}

\begin{proof} Fix $\epsilon > 0$ and $c \in (\epsilon, 1 - \epsilon)$. Let $k = c \log x/\log_2 x$. We evaluate $\pi(x, k)$ using Theorem $2.6$. First, note that
\[\exp\left(O\left(\frac{k}{(\log_2 x)^{1/4}}\right)\right) = \exp\left(O\left(\frac{\log x}{(\log_2 x)^{5/4}}\right)\right) = \exp\left(o\left(\frac{\log x}{\log_2 x}\right)\right),\]
making it negligible.

We have
\[x^\alpha = x\exp\left((H(c) + o(1)) \frac{\log x}{\log_2 x}\right),\]
\[e^{\alpha k} = \exp\left((c + o(1)) \frac{\log x}{\log_2 x}\right),\]
\[(\log x)^k = \exp(k \log_2 x) = \exp(c\log x) = x^c,\]
\[(\log x)^{\alpha k} = x^{\alpha c} = x^c \exp\left((cH(c) + o(1)) \frac{\log x}{\log_2 x}\right).\]
We still need to evaluate $S(\alpha)^{k\alpha}$. Note that $\alpha = 1 + o(1)$. For $t$ sufficiently close to $1$,
\[\sin(\pi/t) = \pi(t - 1) + O((t - 1)^2).\]
Therefore,
\[S(\alpha) = \frac{\alpha}{\pi} (\pi(\alpha - 1) + O((\alpha - 1)^2)) = \alpha(\alpha - 1) + O(\alpha(\alpha - 1)^2) = (H(c) + o(1))\frac{1}{\log_2 x}.\]
So,
\[S(\alpha)^k = \exp\left(-c \frac{\log x \log_3 x}{\log_2 x} + (c \log H(c) + o(1)) \frac{\log x}{\log_2 x}\right).\]
Because $\alpha = 1 + O(1/\log_2 x)$,
\[S(\alpha)^{\alpha k} = \exp\left(-c \frac{\log x \log_3 x}{\log_2 x} + (c \log H(c) + o(1)) \frac{\log x}{\log_2 x}\right).\]
Putting all this together gives us
\[\pi(x, k) = x^{1 - c} \exp\left(c\frac{\log x \log_3 x}{\log_2 x} + (c + H(c) - cH(c) - c \log H(c) + o(1)) \frac{\log x}{\log_2 x}\right).\]

We also calculate the corresponding ordered Bell number term:
\[a_k = \exp(k \log k - (1 + \log_2 2 + o(1))k).\]
Here,
\[k \log k = c\frac{\log x}{\log_2 x} (\log_2 x - \log_3 x) = c \log x - c \frac{\log x \log_3 x}{\log_2 x},\]
\[(1 + \log_2 2 + o(1))k = (c(1 + \log_2 2) + o(1)) \frac{\log x}{\log_2 x}.\]
Thus,
\[a_k = x^c \exp\left(-c\frac{\log x \log_3 x}{\log_2 x} - (c(1 + \log_2 2) + o(1)) \frac{\log x}{\log_2 x}\right).\]

We now have
\[a_k \pi(x, k) = x\exp\left((H(c) - c(H(c) + \log H(c) + \log_2 2) + o(1)) \frac{\log x}{\log_2 x}\right).\]
We maximize the coefficient of $\log x/\log_2 x$ by determining when its derivative is zero. The derivative is
\[H'(c) - H(c) - \log H(c) - \log_2 2 - cH'(c) - (cH'(c)/H(c)).\]
However, $H'(c)$ satisfies the following differential equation \cite[p. $9$]{Ke}:
\[\frac{1}{H'(c)} = \frac{c}{H(c)} + c - 1.\]
Therefore, we want to solve
\[H(c) + \log H(c) = -(1 + \log_2 2).\]
This occurs at
\[H(c) = W(1/(e \log 2)).\]
In order to find $c$, we simply take $H^{-1}(H(c))$ and obtain
\[c = H^{-1}(w) = -we^w \mathrm{Ei}(-w)\]
with $w = W(1/(e \log 2))$. Therefore,
\[a_k \pi(x, k) = x\exp\left((w + we^w \mathrm{Ei}(-w)(w + \log w + \log_2 2) + o(1)) \frac{\log x}{\log_2 x}\right). \qedhere\]
\end{proof}

\section{Negative moments of $F$ and $f$}

We establish the negative moments of $F$, then show that our argument holds for $f$ as well.

\begin{customthm}{1.9} For $\beta > 0$,
\[\sum_{n \leq x} F(n)^{-\beta} = \frac{x}{\log x} \exp((1 + o(1))((1 + \beta)(\log_2 x)(\log_3 x)^\beta)^{1/(1 + \beta)}).\]
\end{customthm}

For the negative moments of $F$, we must refine our techniques. In Section $3$, we showed that if $\beta$ is positive, then
\[\sum_{n \leq x} F(n)^\beta = \exp(O(\log_2 x)) \max_k B_k^\beta \pi(x, k).\]
The same argument holds for $F(n)^{-\beta}$. However, the $\exp(O(\log_2 x))$ term is now larger than our error term. In order to handle this issue, we must obtain upper and lower bounds for
\[\sum_{n \leq x} F(n)^{-\beta} = \sum_k B_k^{-\beta} \pi(x, k)\]
separately.

\subsection{Upper bound}

Recall that the Hardy-Ramanujan Inequality states that there exists a constant $C$ such that
\[\pi(x, k) \ll \frac{x}{\log x} \cdot \frac{(\log_2 x + C)^{k - 1}}{(k - 1)!}.\]
Therefore,
\[\sum_{n \leq x} F(n)^{-\beta} \ll \frac{x}{\log x} \sum_k B_k^{-\beta} \frac{(\log_2 x + C)^{k - 1}}{(k - 1)!}.\]
Let $b_k (x) = B_k^{-\beta} (\log_2 x + C)^{k - 1}/((k - 1)!)$ so that
\[\sum_{n \leq x} F(n)^{-\beta} \ll \frac{x}{\log x} \sum_k b_k (x).\]

We maximize $b_k (x)$ by determining the value of $k$ for which
\[\frac{b_{k + 1} (x)}{b_k (x)} = 1.\]
By Corollary $3.6$,
\[\frac{b_{k + 1} (x)}{b_k (x)} = \left(\frac{B_{k + 1}}{B_k}\right)^{-\beta} \frac{\log_2 x + C}{k} \sim \left(\frac{\log k}{k}\right)^\beta \frac{\log_2 x}{k} = (\log_2 x) \frac{(\log k)^\beta}{k^{1 + \beta}}.\]
This ratio is $1$ when
\[k \sim \left(\frac{1}{(1 + \beta)^\beta} (\log_2 x) (\log_3 x)^\beta\right)^{1/(1 + \beta)}.\]
Call this quantity $k^*$. Note that the maximum value of $b_k (x)$ is $e^{o(k^*)} b_{k^*} (x)$. We also have $b_{k + 1} (x)/b_k (x) < 1/2$ for $k > 2^{1/(1 + \beta)} k^*$. Therefore,
\begin{eqnarray*}
\sum_{n \leq x} F(n)^{-\beta} & \ll & \frac{x}{\log x} \sum_k b_k (x) \\
& = & \frac{x}{\log x} \left(\sum_{k \leq 2^{1/(1 + \beta)} k^*} b_k (x) + \sum_{k > 2^{1/(1 + \beta)} k^*} b_k (x)\right) \\
& \leq & \frac{x}{\log x} (e^{o(k^*)}  2^{1/(1 + \beta)} k^* b_{k^*} (x) + b_{k^*} (x)) \\
& \ll & \frac{x}{\log x} e^{o(k^*)} b_{k^*} (x).
\end{eqnarray*}

Note that
\[b_{k^*} (x) = B_{k^*}^{-\beta} \frac{(\log_2 x + C)^{k^* - 1}}{(k^* - 1)!}.\]
We have
\begin{eqnarray*}
(\log_2 x + C)^{k^* - 1} & = & \exp((k^* - 1)\log(\log_2 x + C)) \\
& = & \exp(k^* \log(\log_2 x + C) + O(\log_3 x)) \\
& = & \exp\left(k^* \log_3 x + O\left(\frac{k^*}{\log_2 x} + \log_3 x\right)\right) \\
& = & \exp(k^* \log_3 x + o(k^*)), \\
(k^* - 1)! & = & \exp(k^* \log k^* - (1 + o(1)) k^*), \\
B_{k^*} & = & \exp(k^* \log k^* - k^* \log_2 k^* - (1 + o(1)) k^*).
\end{eqnarray*}
Therefore,
\begin{eqnarray*}
b_{k^*} (x) & = & \exp(k^* (\log_3 x - (\log k^* - 1 + o(1)) - \beta(\log k^* - \log_2 k^* - 1 + o(1)))) \\
& = & \exp(k^* (\log_3 x - (1 + \beta) \log k^* + \beta \log_2 k^* + (1 + \beta) + o(1))).
\end{eqnarray*}
Note that
\[\log k^* = \frac{1}{1 + \beta} \log_3 x + \frac{\beta}{1 + \beta} \log_4 x - \frac{\beta}{1 + \beta} \log(1 + \beta),\]
\[\log_2 k^* = \log_4 x - \log(1 + \beta) + o(1).\]
Hence,
\[b_{k^*} (x) = \exp((1 + o(1))(1 + \beta)k^*).\]
Substituting our expression for $k^*$ into this equation and multiplying it by $x/\log x$ gives us our desired upper bound.

\subsection{Lower bound and $f$}

We obtain a lower bound with the same formula as our upper bound. Theorem \ref{Sathe} states that if $k \ll \log_2 x$, then
\[\pi' (x, k) \gg \frac{x}{\log x} \cdot \frac{(\log_2 x)^{k - 1}}{(k - 1)!},\]
where $\pi'(x, k)$ is the number of squarefree $n \leq x$ satisfying $\omega(n) = k$. If we plug $k^*$ into this formula, we have the correct bound.

\begin{customthm}{1.9} For $\beta > 0$,
\[\sum_{n \leq x} f(n)^{-\beta} = \frac{x}{\log x} \exp((1 + o(1))((1 + \beta)(\log_2 x)(\log_3 x)^\beta)^{1/(1 + \beta)}).\]
\end{customthm}

\begin{proof}
By definition, $f(n) \geq F(n)$ for all $n$. Therefore,
\[\sum_{n \leq x} f(n)^{-\beta} \leq \sum_{n \leq x} F(n)^{-\beta}.\]
But by restricting the sum on the left-hand side to squarefree $n$ we have
\[\sum_{n\leq x} f(n)^{-\beta} \geq \sum_{\substack{n\leq x \\ n \text{ squarefree}}} f(n)^{-\beta} = \sum_{\substack{n \leq x \\ n \text{ squarefree}}} F(n)^{-\beta}.\]
but we have already shown the correct lower bound for the rightmost sum.
\end{proof}

\section{Negative moments of $G$ and $g$}

We establish upper bounds for the negative moments of $G$ using the same techniques that we used in the previous section. Arguments similar to those found in the previous subsection establish that these bounds are the correct asymptotic formulae for the negative moments of $G$ and $g$.

\begin{customthm}{1.10} For $\beta > 0$, the sums
\[\sum_{n \leq x} G(n)^{-\beta}, \quad \sum_{n \leq x} g(n)^{-\beta}\]
are both
\[\frac{x}{\log x} \exp((1 + o(1))(1 + \beta)(\log 2)^{\beta/(1 + \beta)} (\log_2 x)^{1/(1 + \beta)}).\]
\end{customthm}

\begin{proof} We now have
\[\sum_{n \leq x} G(n)^{-\beta} \ll \frac{x}{\log x} \sum_k a_k^{-\beta} \frac{(\log_2 x + C)^{k - 1}}{(k - 1)!},\]
for some constant $C$. Let $A_k (x) = a_k^{-\beta} (\log_2 x + C)^{k - 1}/((k - 1)!)$ so that
\[\sum_{n \leq x} G(n)^{-\beta} \ll \frac{x}{\log x} \sum_k A_k (x).\]

We determine when $A_{k + 1} (x)/A_k (x) \sim 1$. Recall that
\[a_k \sim \frac{1}{2 \log 2} \cdot \frac{k!}{(\log 2)^k}.\]
Hence,
\[\frac{A_{k + 1} (x)}{A_k (x)} \sim \left(\frac{\log 2}{k}\right)^\beta \frac{\log_2 x}{k} = (\log 2)^\beta \frac{\log_2 x}{k^{1 + \beta}}.\]
This function is equal to $1$ when
\[k = (\log 2)^{\beta/(1 + \beta)} (\log_2 x)^{1/(1 + \beta)}.\]
Call this quantity $k^*$.

Note that if $k > 2^{1/(1 + \beta)} k^*$, then $A_k (x) < A_{k^*}(x) /2$. In this case, we simply need to determine $A_{k^*} (x)$. We have
\begin{eqnarray*}
(\log_2 x + C)^{k^* - 1} & = & \exp(k^* \log_3 x + o(k^*)), \\
(k^* - 1)! & = & \exp(k^* \log k^* - (1 + o(1)) k^*), \\
a_{k^*} & = & \exp(k^* \log k^* - (1 + \log_2 2 + o(1)) k^*).
\end{eqnarray*}
Therefore,
\[A_{k^*} (x) = \exp(k^* (\log_3 x - (1 + \beta) \log k^* + \beta (1 + \log_2 2) + 1 + o(1))).\]
In this case,
\[\log k^* = \frac{1}{1 + \beta} \log_3 x + \frac{\beta \log_2 2}{1 + \beta},\]
which implies that
\[A_{k^*} (x) = \exp(k^* (1 + \beta + o(1))),\]
establishing our result.
\end{proof}

\section*{Acknowledgments}

The first author (M.J.) was partially supported by the Research and Training Group grant DMS-1344994 funded by the National Science Foundation. We thank Paul Pollack for helpful comments.

\end{document}